\documentclass[12pt]{amsart} 

\usepackage{amsmath} \usepackage{amssymb} \usepackage{amsthm}
\usepackage{amscd} 
\usepackage[all]{xy} 
\usepackage{enumerate}
\usepackage{mathrsfs}

\def\a{{\mathfrak{a}}} 
\def\b{{\mathfrak{b}}}

\def\J{{\mathcal{J}}} 
\def\m{{\mathfrak{m}}}
 
\def\Z{{\mathbb{Z}}}

\def\sO{{\mathcal{O}}}

\def\Q{{\mathbb{Q}}}

\def\Hom{{\mathrm{Hom}}}

\def\Im{{\mathrm{Im}}} 
 
\def\Ann{{\mathrm{Ann}}}
\def\Spec{{\mathrm{Spec\; }}}

\def\J{{\mathcal{J}}}

\theoremstyle{plain}
\newtheorem{thm}{Theorem}[section] 

\newtheorem{prop}[thm]{Proposition}

\newtheorem*{mainthm}{Theorem}

\newtheorem{lem}[thm]{Lemma}
\theoremstyle{definition} 
\newtheorem{defn}[thm]{Definition}
\newtheorem{propdef}[thm]{Proposition-Definition} 
 
\theoremstyle{remark}
\newtheorem{rem}[thm]{Remark}

\newtheorem*{cl}{Claim}

\newtheorem*{acknowledgement}{Acknowledgments}

\title{A subadditivity formula for multiplier ideals\\ associated to log pairs}
\author{Shunsuke Takagi}
\address{Department of Mathematics, Kyushu University, 744 Motooka, Nishi-ku, Fukuoka 819-0395, Japan.}
\email{stakagi@math.kyushu-u.ac.jp}

\begin{document}
\tolerance = 9999

\maketitle
\markboth{SHUNSUKE TAKAGI}{A SUBADDITIVITY FORMULA FOR MULTIPLIER IDEALS 
}

\begin{abstract}
As a generalization of formulas given in \cite{DEL}, \cite{Ei} and \cite{Ta1}, 
we prove a subadditivity formula for multiplier ideals associated to log pairs. 
\end{abstract}

\section*{Introduction}
Multiplier ideals satisfy vanishing theorems, making them a fundamental tool in higher-dimensional algebraic geometry. 
They are defined as follows:
let $(X, \Delta)$ be a \textit{log pair}, that is, $\Delta$ be an effective $\Q$-divisor on a normal variety $X$ over a field of characteristic zero such that $K_X+\Delta$ is $\Q$-Cartier. 
Let $\a \subseteq \sO_X$ be an ideal sheaf and $t$ be a real number. 
Suppose that $\pi:\widetilde{X} \to X$ is a log resolution of $(X, \Delta, \a)$, that is, 
$\pi$ is a proper birational morphism with $\widetilde{X}$ nonsingular such that $\a \sO_{\widetilde{X}}=\sO_{\widetilde{X}}(-F)$ is invertible and $\mathrm{Exc}(\pi) \cup \mathrm{Supp}(\pi^{-1}_*\Delta) \cup \mathrm{Supp}(F)$ is a simple normal crossing divisor. 
Then the multiplier ideal $\J((X, \Delta);\a^t)$ of $\a$ with exponent $t$ for the pair $(X, \Delta)$ is
$$\J((X, \Delta);\a^t)=\pi_*\sO_{\widetilde{X}}(\lceil K_{\widetilde{X}}-\pi^*(K_X+\Delta)-tF \rceil) \subseteq \sO_X.$$

Demailly, Ein and Lazarsfeld \cite{DEL} formulated a subadditivity property of multiplier ideals on nonsingular varieties, which states that the multiplier ideal of the product of two ideal sheaves is contained in the product of their individual multiplier ideals.
Their formula has many interesting applications in algebraic geometry and commutative algebra, such as Fujita's approximation theorem (see \cite{Fu} and \cite[Theorem 10.3.5]{La}) and its local analogue (see \cite{ELS2}), a problem on the growth of symbolic powers of ideals in regular rings (see \cite{ELS}), and etc. 
Later, Takagi \cite{Ta1} and Eisenstein \cite{Ei} generalized their formula to the case of $\Q$-Gorenstein varieties, that is, the case when $\Delta=0$ in the above definition of multiplier ideals. 
In this article, we study a further generalization to the case of log pairs, when the importance of multiplier ideals is particularly highlighted. 
The following is our main result. 

\begin{mainthm}[Theorems \ref{main} and \ref{main2}]
Let $X$ be a normal variety over an algebraically closed field of characteristic zero and $\Delta$ be an effective $\Q$-divisor on $X$ such that $r(K_X+\Delta)$ is Cartier for some integer $r \ge 1$. 
Let $\mathrm{Jac}_X$ denote the Jacobian ideal sheaf of $X$.
Then 
$$\mathrm{Jac}_X \cdot \J\left((X, \Delta); \a^s\b^t \sO_X(-r \Delta)^{1/r}\right) \subseteq  \J((X, \Delta); \a^s)\J((X, \Delta); \b^t).$$
for any ideal sheaves $\a, \b \subseteq \sO_X$ and for any real numbers $s, t>0$. 
\end{mainthm}

We give two proofs of this. 
The first proof is a refinement of the arguments in \cite{Ta1}. 
We give a subadditivity formula for generalized test ideals, which itself is interesting from the point of view of algebraic geometry and commutative algebra in positive characteristic. 
Then we use a correspondence between multiplier ideals and big generalized test ideals (see \cite{HY} and \cite{Ta2}) to obtain the assertion. 
In the second proof, we employ the same methods as those used in \cite{Ei}. 
We pull back the problem to the product $X \times X$ and then the desired formula on $X$ is obtained by restricting to the diagonal.
We use factorizing embedded resolutions to compute the restriction of multiplier ideals on $X \times X$ to the diagonal. 

%An interesting application of our formula is found in \cite{BdFF}.

\begin{small}
\begin{acknowledgement}
The author would like to thank Tommaso de Fernex who raised this problem. 
He is indebted to Salvatore Cacciola for pointing out a mistake in a previous version of the article. 
The author also would like to express his gratitude to the Massachusetts Institute of Technology, where a part of this work was done, for their hospitality during the winter of 2010--2011. 
The author was partially supported by Grant-in-Aid for Young Scientists (B) 20740019 from JSPS and by Program for Improvement of Research Environment for Young Researchers from SCF commissioned by MEXT of Japan.
\end{acknowledgement}
\end{small}

\section{Preliminaries on big generalized test ideals}
In this section, we briefly review the definition and basic properties of big generalized test ideals, which we will need later.
The reader is referred to \cite{HT}, \cite{HY}, \cite{Ta2} and \cite{BSTZ} for details. 
The reader interested only in an algebro-geometric proof of our result can directly go to Section 3. 

Throughout this paper, all schemes are Noetherian, excellent and separated.  
A \textit{graded family} of ideal sheaves $\a_{\bullet}=\{\a_m\}_{m \ge 0}$ on an integral scheme $X$ means a collection of nonzero ideal sheaves $\a_m \subseteq \sO_X$, satisfying $\a_0=\sO_X$ and $\a_k  \a_l \subseteq \a_{k+l}$ for all $k, l \ge 1$. 
For example, given an ideal sheaf $\a \subseteq \sO_X$ and a real number $t \ge 0$, $\a_{\bullet}=\{\a^{\lceil tm \rceil}\}$ is a graded family of ideal sheaves on $X$.  
Another example of graded families of ideal sheaves is $I_{\Delta}^{(\bullet)}=\{\sO_X(-\lceil m \Delta \rceil)\}$ where $\Delta$ is an effective $\Q$-divisor on a normal scheme $X$. 

Let $X$ be an integral scheme of prime characteristic $p$. 
For each integer $e \ge 1$,  we denote by $F^e: X \to X$ or $F^e: \sO_X \to F^e_*\sO_X$ the $e$-th iteration of the absolute Frobenius morphism on $X$. 
We say that $X$ is \textit{$F$-finite} if $F: X \to X$ is a finite morphism. 
For example, a field $K$ of characteristic $p>0$ is $F$-finite if and only if $[K:K^p]$ is finite. 
Given an ideal sheaf $I \subseteq \sO_X$, for each $q=p^e$, we denote by $I^{[q]} \subseteq \sO_X$ the ideal sheaf identified with $I \cdot F^e_*\sO_X$ via the identification $F^e_*\sO_X \cong \sO_X$. 
generated by the $q$-th powers of all elements of $I$. 

We give the definition of big generalized test ideals, using a generalization of tight closure \cite{HY}, \cite{Ta2}. 
First we recall the definition of a generalization of tight closure.  

\begin{defn}[\textup{\cite[Definition 6.1]{HY}, \cite[Definition 2.1]{Ta2}, \cite[Definition 2.16]{Sc}}]\label{tight closure def}
Let $X$ be a $d$-dimensional $F$-finite normal integral affine scheme of characteristic $p>0$, $\Delta$ be an effective $\Q$-divisor on $X$ and $\a_{\bullet}$ be a graded family of ideals on $X$.  
\renewcommand{\labelenumi}{(\roman{enumi})}
\begin{enumerate}
\item 
Let $I \subseteq \sO_{X}$ be a nonzero ideal. Then the \textit{$(\Delta, \a_{\bullet})$-tight closure} $I^{*(\Delta, \a_{\bullet})}$ of $I$ is defined to be the ideal of $\sO_{X}$ consisting of all $z \in \sO_{X}$ for which there exists a nonzero element $c \in \sO_{X}$ such that
$$c\a_{q-1}z^q \in I^{[q]}\sO_X(\lceil (q-1) \Delta \rceil)$$ 
for all large $q=p^e$. 

\item 
Denote by $E=\bigoplus_x H^d_x(\omega_X)$ the direct sum, taken over all closed points $x \in X$, of the $d$-th local cohomology modules of the canonical module $\omega_X$ of $X$ with support on $x$. 
For each integer $e \ge 1$, let 
$$F^e_E:E=\bigoplus_x H^d_x(\sO_X(K_X)) \to \bigoplus_x H^d_x(\sO_X(p^eK_X))$$ 
be the map induced by the $e$-times iterated Frobenius map $F^e: \sO_{X} \to F^e_*\sO_{X}$.  
Then the \textit{$(\Delta, \a_{\bullet})$-tight closure} $0^{*(\Delta, \a_{\bullet})}_{E}$ of the zero submodule in $E$ is defined to be the submodule of $E$ consisting of all $z \in E$ for which there exists  a nonzero element $c \in \sO_{X}$ such that 
$$c\a_{q-1}F^e_E(z)=0 \; \textup{ in $\; \bigoplus_x H^d_x(\sO_X(qK_X+\lceil (q-1) \Delta \rceil))$}$$ 
for all large $q=p^e$. 

\item We say that a nonzero element $c \in \sO_X$  is a \textit{big sharp test element} for the triple $(X, \Delta, \a_{\bullet})$ if for all $z \in 0^{*(\Delta, \a^t)}_{E}$, we have 
$$c\a_{q-1}F^e_E(z)=0  \; \textup{ in $\; \bigoplus_x H^d_x(\sO_X(qK_X+\lceil (q-1) \Delta \rceil))$}$$
for every $q=p^e$. 
Big sharp test elements always exist (see \cite[Lemma 2.17]{Sc}). 
\end{enumerate}
\end{defn}

\begin{propdef}[\textup{\cite[Definition-Proposition 3.3]{BSTZ}, cf.~\cite[Lemma 2.1]{HT}}]\label{characterization}
Let the notation be the same as in Definition \ref{tight closure def}.
Then each of the following conditions defines the same ideal, which is called the \textit{big generalized
test ideal} for the triple $(X, \Delta, \a_{\bullet})$ and denoted by $\tau_b(X, \Delta, \a_{\bullet})$. 
\renewcommand{\labelenumi}{(\alph{enumi})}
\begin{enumerate}
\item $\Ann_{\sO_{X}}0^{*(\Delta, \a_{\bullet})}_{E}$.
\item The ideal generated by all big sharp test elements for $(X, \Delta, \a_{\bullet})$. 
\item For any integer $e_0 \ge 1$, the sum 
$$\sum_{e \ge e_0}\sum_{\phi_e} \phi_e(F^e_*(c \a_{p^e-1})),$$
where $\phi_e$ ranges over all elements of $\Hom_{\sO_X}(F^e_*\sO_{X}((\lceil (p^e-1) \Delta \rceil), \sO_X)$ and $c$ is a big sharp test element for $(X, \Delta, \a_{\bullet})$.
\end{enumerate}
When $\Delta=0$, we denote this ideal simply by $\tau_b(X, \a_{\bullet})$. 
When $\a_{\bullet}=\{\a^{\lceil tm \rceil}\}$ for a nonzero ideal $\a \subseteq \sO_X$ and a real number $t>0$, we denote this ideal by $\tau_b(X, \Delta, \a^t)$. 
\end{propdef}

\begin{rem}\label{remark}
(1)
Given graded families of ideals $\a_{1, \bullet}, \dots, \a_{r, \bullet}$ on $X$, we can define the ideal 
$\tau_b(X, \Delta, \a_{1, \bullet} \cdots \a_{r, \bullet})$ in the same manner as above. 

(2) (\cite[Remark 1.4]{TY})
$\tau_b(X, \Delta, \a_{\bullet})$ is equal to the unique maximal element among the set of ideals $\{\tau_b(X, \Delta, \a_{p^e}^{1/p^e})\}_{e \ge 0}$ with respect to inclusion. 
If $\a_{\bullet}$ is a descending filtration, then $\tau_b(X, \Delta, \a_{\bullet})$ is equal to the unique maximal element among the set of ideals $\{\tau_b(X, \Delta, \a_{m}^{1/m})\}$. 

(3)
Since the formation of $\tau_b(X, \Delta, \a_{\bullet})$ commutes with localization (see \cite[Proposition 3.1]{HT}), we can define the ideal sheaf $\tau_b(X, \Delta, \a_{\bullet})$ when $X$ is a non-affine scheme by gluing over affine charts. 
\end{rem}

Hara--Yoshida \cite{HY} and Takagi \cite{Ta2} proved a correspondence between multiplier ideals and big generalized test ideals. 
In order to state their results, we briefly recall how to reduce things from characteristic zero to characteristic $p>0$. 
We refer the reader to \cite[Chapter 2]{HH} and \cite[Section 3.2]{MS} for details. 

Let $\Delta$ be an effective $\Q$-divisor on a normal variety $X$ over a field $k$ of characteristic zero. 
Let $\a \subseteq \sO_X$ be an ideal sheaf and $t>0$ be a real number. 
Then a \textit{model} of $(X, \Delta, \a)$ over a finitely generated $\Z$-subalgebra $A$ of $k$ is a triple $(X_A, \Delta_A, \a_A)$ of a normal integral scheme $X_A$ of finite type over $A$, an effective $\Q$-divisor $\Delta_A$ on $X_A$ and an ideal sheaf $\a_A \subseteq \sO_{X_A}$ such that $X_A \times_{\Spec A} \Spec k \cong X$, $\rho^*\Delta_A = \Delta$ and $\rho^{-1}\a_A=\a$, where $\rho: X \to X_A$ is a natural projection. 
Given a closed point $\mu \in \Spec A$, we denote by $X_{\mu}$ (resp., $\Delta_{\mu}$, $\a_{\mu}$) the fiber of $X_A$ (resp., $\Delta_A$, $\a_A$) over $\mu$.  
Note that $X_{\mu}$ is a scheme of finite type over the residue field $\kappa(\mu)$ of $\mu$, which is a finite field.

\begin{thm}[\textup{\cite[Theorem 3.2]{Ta2}, \cite[Theorem 6.8]{HY}}]\label{mult thm}
Let $X$ be a normal variety over a field $k$ of characteristic zero and $\Delta$ be an effective $\Q$-divisor on $X$ such that $K_X+\Delta$ is $\Q$-Cartier. 
Let $\a$ be a nonzero ideal sheaf on $X$ and $t>0$ be a real number. 
Given any model $(X_A, \Delta_A, \a_A)$ over a finitely generated $\Z$-subalgebra $A$ of $k$, 
there exists an open subset $W \subseteq \Spec A$ $($depending on $t$$)$ such that  
$$\J((X,\Delta); \a^t)_{\mu}=\tau_b(X_{\mu},\Delta_{\mu},\a_{\mu}^t)$$
for every closed point $\mu \in W$. 
\end{thm}

\section{A proof using big generalized test ideals}
In this section, we will give a subadditivity formula for multiplier ideals associated to log pairs,  using big generalized test ideals. 
We start with the following lemma. 
\begin{lem}\label{key lemma}
Let $X$ be an $F$-finite normal integral affine scheme of characteristic $p>0$ and $\Delta$ be an effective $\Q$-divisor on $X$.
Let $\a, \b \subseteq \sO_X$ be ideals and $s, t>0$ be real numbers.  
\renewcommand{\labelenumi}{$(\arabic{enumi})$}
\begin{enumerate}
\item For each integer $e \ge 1$, 
$$\tau_b(X, \a^s) F^e_*(\sO_X(-\lceil p^e \Delta \rceil))  \subseteq \tau_b(X, \Delta, \a^s) F^e_*\sO_X.$$ 
\item 
Let $\mathcal{I}_{\Delta}^{(\bullet)}=\{\sO_X(-\lceil m \Delta \rceil)\}$ be the graded family of ideals associated to $\Delta$. 
Then one has 
$$\tau_b(X, \a^s)^{*\a^s} \tau_b\left(X, \Delta, \a^s\b^t\mathcal{I}_{X}^{(\bullet)}\right) \subseteq \tau_b(X, \Delta, \a^s)\tau_b(X, \Delta, \b^t)$$
\end{enumerate}
\end{lem}

\begin{proof}
(1) Let $c \in \sO_X$ be a big sharp test element for both $(X, \a^t)$ and $(X, \Delta, \a^t)$. 
By Proposition-Definition \ref{characterization}, 
\begin{align*}
&\tau_b(X, \a^s) F^{e}_*(\sO_X(-\lceil p^{e}\Delta \rceil))\\
=&\sum_{e' \ge e}\sum_{\phi_{e'}}F^e_*(\sO_X(-\lceil p^e\Delta \rceil))\phi_{e'}(F^{e'}_*(c \a^{\lceil t(p^{e'}-1) \rceil})),
\end{align*}
where $\phi_{e'}$ ranges over all elements of $\Hom_{\sO_X}(F^{e'}_*\sO_{X}, \sO_X)$. 
For all elements $s \in \sO_X(-\lceil p^{e}\Delta \rceil)$, since $s^{p^{e'-e}} \in \sO_X(-\lceil p^{e'}\Delta \rceil)$,  
$$F^e_*s \cdot \phi_{e'}: F^{e'}_*\sO_X(\lceil (p^{e'}-1) \Delta \rceil) \xrightarrow{F^e_*s} F^{e'}_*\sO_X \xrightarrow{\phi_{e'}} \sO_X$$
is viewed as an element of $\Hom_{\sO_X}(F^{e'}_*\sO_{X}(\lceil (p^{e'}-1) \Delta \rceil), \sO_X)$.
Thus, applying Proposition-Definition \ref{characterization} again, one has 
\begin{align*}
&\sum_{e' \ge e}\sum_{\phi_{e'}}F^e_*(\sO_X(-\lceil p^e\Delta \rceil))\phi_{e'}(F^{e'}_*(c \a^{\lceil t(p^{e'}-1) \rceil}))\\
\subseteq & \sum_{e' \ge e}\sum_{\psi_{e'}}\psi_{e'}(F^{e'}_*(c \a^{\lceil t(p^{e'}-1) \rceil}))\\
=& \tau_b(X, \Delta, \a^t),
\end{align*}
where $\psi_{e'}$ ranges over all elements of $\Hom_{\sO_X}(F^{e'}_*\sO_{X}(\lceil (p^{e'}-1) \Delta \rceil), \sO_X)$. 

 (2)
 Since the formation of big generalized test ideals commutes with localization and completion (see \cite[Propositions 3.1 and 3.2]{HT}), we may assume that $(X, x)=\Spec R$, where $(R, \m)$ is a $d$-dimensional complete normal local ring of characteristic $p>0$. 
 Let $E=H^d_x(\omega_X)$ be the $d$-th local cohomology module of $\omega_X$ with support on $x$ and let $F^e_E:E=H^d_x(\sO_X(K_X)) \to  H^d_x(\sO_X(p^eK_X))$
be the map induced by the $e$-times iterated Frobenius map $F^e: \sO_{X} \to F^e_*\sO_{X}$.  
 Then by local duality, the assertion is equivalent to saying that 
$$\left(0^{*\left(\Delta, \a^s\b^t\mathcal{I}_{\Delta}^{(\bullet)}\right)}_E \colon \tau_b(X, \a^s)^{*\a^s}\right)_E \supseteq \left(0^{*(\Delta, \b^t)}_E: \tau_b(X, \Delta, \a^s)\right)_E.$$
Let $z \in \left(0^{*(\Delta, \b^t)}_E: \tau_b(R, \Delta, \a^s)\right)_E$. 
Then there exists a nonzero element $c \in \sO_X$ such that 
$$c\b^{\lceil t(q-1) \rceil}\tau_b(X, \Delta, \a^s)^{[q]}F^e_E(z)=0\textup{ in }H^d_x(\sO_X(qK_X+\lceil (q-1) \Delta \rceil))$$ for all large $q=p^e$. 
Fix any nonzero element $\delta \in \sO_X(-\lceil \Delta \rceil)$. 
 By the definition of $\a^s$-tight closure and (1), there exists another nonzero element $c' \in \sO_X$ such that 
\begin{align*}
c' \delta \a^{\lceil s(q-1) \rceil}\sO_X(-\lceil (q-1) \Delta \rceil) (\tau_b(X, \a^s)^{*\a^s})^{[q]} & \subseteq 
c' \a^{\lceil s(q-1) \rceil}\sO_X(-\lceil q \Delta \rceil) (\tau_b(X, \a^s)^{*\a^s})^{[q]}\\ 
&\subseteq \sO_X(-\lceil q \Delta \rceil)\tau_b(X, \a^s)^{[q]}\\
&\subseteq \tau_b(X, \Delta, \a^s)^{[q]}
\end{align*}
for all large $q=p^e$. 
Therefore, one has 
$$cc'\delta \a^{\lceil s(q-1) \rceil} \b^{\lceil t(q-1) \rceil}\sO_X(-\lceil (q-1) \Delta \rceil)(\tau_b(X, \a^s)^{*\a^s})^{[q]}F^e_E(z)=0$$
in $H^d_x(\sO_X(qK_X+\lceil (q-1) \Delta \rceil))$ for all large $q=p^e$. 
That is, $\tau_b(X, \a^s)^{*\a^s}z \subseteq 0^{*\left(\Delta, \a^s\b^t, \mathcal{I}_{\Delta}^{(\bullet)}\right)}_E$. 
\end{proof}

As a consequence of the above lemma, we obtain a subadditivity formula for big generalized test ideals. We stress that $K_X+\Delta$ is not necessarily $\Q$-Cartier in Proposition \ref{charp}. 

\begin{prop}\label{charp}
Let $X$ be a normal integral scheme essentially of finite type over an F-finite field and $\Delta$ be an effective $\Q$-divisor on $X$. 
Let $\mathcal{I}_{\Delta}^{(\bullet)}=\{\sO_X(-\lceil m \Delta \rceil)\}$ denote the graded family of ideal sheaves associated to $\Delta$ and $\mathrm{Jac}_X$ denote the Jacobian ideal sheaf of $X$.
Then 
$$\mathrm{Jac}_X \cdot \tau_b\left(X, \Delta, \a^s\b^t \mathcal{I}_{\Delta}^{(\bullet)}\right) \subseteq \tau_b(X, \Delta, \a^s)\tau_b(X, \Delta, \b^t)$$
for any ideal sheaves $\a, \b \subseteq \sO_X$ and for any real numbers $s, t>0$. 
\end{prop}
\begin{proof}
The question is local, so we may assume that $X$ is affine. 
Since $\mathrm{Jac}_X\subseteq \tau_b(X, \a^s)^{*\a^s}$ by \cite[Lemma 2.6]{Ta1}, 
the assertion immediately follows from Lemma \ref{key lemma} (2). 
\end{proof}

Before formulating a subadditivity property of multiplier ideals, we recall the definition of asymptotic multiplier ideal sheaves. 
Let $\Delta$ be an effective $\Q$-divisor on a normal variety $X$ over a field of characteristic zero such that $K_X+\Delta$ is $\Q$-Cartier. 
Let $a_{\bullet}=\{\a_m\}$ be a graded family of ideal sheaves on $X$. 
Then the \textit{asymptotic multiplier ideal sheaf} $\J((X, \Delta); \a_{\bullet})$ of $\a_{\bullet}$ for the pair $(X, \Delta)$ is defined to be the
unique maximal member among the family of ideal sheaves $\{\J((X, \Delta); \a_m^{1/m})\}$ with respect
to inclusion. We refer the reader to \cite[Chapter 10]{La} for details.

\begin{thm}\label{main}
Let $X$ be a normal variety over a field of characteristic zero and $\Delta$ be an effective $\Q$-divisor on $X$ such that $K_X+\Delta$ is $\Q$-Cartier. 
Let $\mathcal{I}_{\Delta}^{(\bullet)}=\{\sO_X(-\lceil m \Delta \rceil)\}$ denote the graded family of ideal sheaves associated to $\Delta$ and let $\mathrm{Jac}_X$ denote the Jacobian ideal sheaf of $X$.
Then
$$\mathrm{Jac}_X \cdot \J\left((X, \Delta); \a^s\b^t\mathcal{I}_{\Delta}^{(\bullet)}\right) \subseteq \J((X, \Delta); \a^s)\J((X, \Delta); \b^t)$$
for any ideal sheaves $\a, \b \subseteq \sO_X$ and for any real numbers $s, t>0$. 
\end{thm}
\begin{proof}
Take sufficiently large and divisible $m$ such that 
$$\J\left((X, \Delta); \a^s\b^t\mathcal{I}_{\Delta}^{(\bullet)}\right)=\J((X, \Delta); \a^s\b^t \sO_X(-\lceil m \Delta \rceil)^{1/m}).$$
It follows from a combination of Remark \ref{remark}, Theorem \ref{mult thm} and Proposition \ref{charp} that for a model $(X_A, \Delta_A, \a_A, \b_A)$ over a finitely generated $\Z$-subalgebra $A$ of $k$, there exists an open subset $W \subseteq \Spec A$ such that  
\begin{align*}
&(\mathrm{Jac}_X)_{\mu} \cdot \J((X, \Delta); \a^s\b^t \sO_X(-\lceil m \Delta \rceil)^{1/m})_{\mu}\\
=&\mathrm{Jac}_{X_{\mu}} \cdot \tau_b(X_{\mu}, \Delta_{\mu}, \a_{\mu}^s\b_{\mu}^t \sO_{X_{\mu}}(-\lceil m \Delta_{\mu} \rceil)^{1/m})\\
\subseteq& \mathrm{Jac}_{X_{\mu}} \cdot \tau_b\left(X_{\mu}, \Delta_{\mu}, \a_{\mu}^s\b_{\mu}^t \mathcal{I}_{\Delta}^{(\bullet)}\right)\\
\subseteq & \tau_b(X_{\mu}, \Delta_{\mu}, \a_{\mu}^s) \cdot \tau_b(X_{\mu}, \Delta_{\mu}, \b_{\mu}^t)\\
=&  \J((X, \Delta); \a^s)_{\mu} \cdot \J((X, \Delta); \b^t)_{\mu}
\end{align*}
for all closed points $\mu \in W$. 
This implies that 
$$\mathrm{Jac}_X \cdot \J((X, \Delta); \a^s\b^t \sO_X(-\lceil m \Delta \rceil)^{1/m}) \subseteq  \J((X, \Delta); \a^s)\J((X, \Delta); \b^t).$$
\end{proof}

\section{An algebro-geometric proof}
In this section, employing the same methods as those used in \cite{Ei}, 
we give an algebro-geometric proof of the above subadditivity formula for multiplier ideals. 
Throughout this section, let $X$ be a $d$-dimensional normal variety over an algebraically closed field of characteristic zero and $\Delta$ be an effective $\Q$-divisor on $X$ such that $K_X+\Delta$ is $\Q$-Cartier. 
For a closed subscheme $Z$ of $X$, we denote by $\mathcal{I}_Z \subseteq \sO_X$ the defining ideal sheaf of $Z$ in $X$. 

First we recall the definition of factorizing embedded resolutions. 
\begin{defn}
Let $Z$ be a reduced closed subscheme of $X$ which is not contained in the singular locus $\mathrm{Sing}(X)$ of $X$. 
A \textit{factorizing embedded resolution} of $Z$ in $X$ is a proper birational morphism $f: \overline{X} \to X$ with $\overline{X}$ smooth such that
\renewcommand{\labelenumi}{(\alph{enumi})}
\begin{enumerate}
\item
$f$ is an isomorphism at every generic point of $Z \subseteq X$, 
\item
the exceptional locus $\mathrm{Exc}(f)$ is a simple normal crossing divisor, 
\item
the strict transform $\overline{Z}$ of $Z$ in $\overline{X}$ is smooth and has simple normal crossings with $\mathrm{Exc}(f)$, 
\item $\mathcal{I}_Z \sO_{\overline{X}}=\mathcal{I}_{\overline{Z}}\sO_{\overline{X}}(-R_Z)$ where $R_Z$ is an $f$-exceptional divisor on $\overline{X}$. 
\end{enumerate}
Such a resolution always exists (see \cite{BEV}). 
\end{defn}

\begin{lem}[\textup{cf.~\cite[Lemma 3.6]{Ei}}]\label{surjective}
Let $\a \subseteq \sO_X$ be an ideal sheaf and $t>0$ be a real number. 
Let $Z \subseteq X$ be a reduced equidimensional closed subscheme of codimension $c$, none of whose components is contained in $\mathrm{Sing}(X) \cup \mathrm{Supp}(\Delta) \cup \mathrm{Supp}(V(\a))$. 
Let $f: \overline{X} \to X$ be a log resolution of $(X, \Delta, \a)$ which is simultaneously a factorizing embedded resolution of $Z$ in $X$ so that $\mathcal{I}_{Z} \sO_{\overline{X}}=\mathcal{I}_{\overline{Z}}\sO_{\overline{X}}(-R_Z)$, where $\overline{Z}$ is the strict transform of $Z$ in $\overline{X}$. 
Put
$$B:=\lceil K_{\overline{X}}-f^*(X_X+\Delta)-t \cdot f^{-1}(V(\a))-c \cdot R_Z\rceil.$$
Then the restriction map
$$f_*\sO_{\overline{X}}(B) \to {f|_{\overline{Z}}}_*\sO_{\overline{Z}}(B|_{\overline{Z}})$$
is surjective. 
\end{lem}
\begin{proof}
It suffices to show that $R^1f_*\left(\mathcal{I}_{\overline{Z}}\sO_{\overline{X}}(B)\right)=0$. 
Let $g: Y \to \overline{X}$ be the blow-up of $\overline{X}$ along $\overline{Z}$ with reduced exceptional divisor $E$, and denote by $h=(g \circ f):Y \to X$ the composite morphism. Since 
\begin{align*}
(\mathcal{I}_{\overline{Z}}\sO_Y)g^*\sO_{\overline{X}}(B)
&=\sO_Y(\lceil g^*K_{\overline{X}}-h^*(K_X+\Delta)-t \cdot h^{-1}(V(\a))-c \cdot g^*R_Z -E\rceil)\\
&=\sO_Y(\lceil K_{Y}-h^*(K_X+\Delta)-t \cdot h^{-1}(V(\a))-c \cdot h^{-1}(X) \rceil),
\end{align*}
it follows from Kawamata--Viehweg vanishing theorem that 
$$R^ih_*\left((\mathcal{I}_{\overline{Z}}\sO_Y)g^*\sO_{\overline{X}}(B) \right)=R^ig_*\left((\mathcal{I}_{\overline{Z}}\sO_Y)g^*\sO_{\overline{X}}(B) \right)=0$$ for all $i>0$. 
We use the Leray spectral sequence to conclude that 
$$R^if_*\left(\mathcal{I}_{\overline{Z}}\sO_{\overline{X}}(B) \right)=R^if_*\left(g_*((\mathcal{I}_{\overline{Z}}\sO_Y)g^*\sO_{\overline{X}}(B))\right)=0$$
for all $i>0$.
\end{proof}

\begin{defn}\label{char 0 def}
Given any positive integer $r$ such that $r(K_X+\Delta)$ is Cartier, consider the natural map 
$$\rho_{r, \Delta}: (\Omega_X^d)^{\otimes r} \to \sO_X(rK_X) \to \sO_X(r(K_X+\Delta)).$$
Let $\mathcal{I}_{r, \Delta} \subseteq \sO_X$ be the ideal sheaf so that $\Im \; \rho_{r, \Delta}=\mathcal{I}_{r, \Delta} \sO_X(r(K_X+\Delta))$. 
Note that if $\mathrm{Jac}_X$ is the Jacobian ideal sheaf of $X$, then $\mathrm{Jac}_X^r \cdot \sO_X(-r\Delta) \subseteq \mathcal{I}_{r, \Delta}$.
\end{defn}

\begin{lem}[\textup{cf.~\cite[Lemma 4.5]{Ei}}]\label{Jacobian}
Let $f: \overline{X} \to X$ be a birational morphism with $X$ smooth and $\mathrm{Jac}_f$ be the Jacobian ideal sheaf of $f$. 
Given an integer $r \ge 1$ such that $r (K_X+\Delta)$ is Cartier, One has
$$\mathrm{Jac}_f^r=\mathcal{I}_{r, \Delta} \sO_{\overline{X}}(-r(K_{\overline{X}}-f^*(K_X+\Delta))).$$
\end{lem}
\begin{proof}
First note that by the definition of $\mathrm{Jac}_f$, the image of the natural map $f^*(\Omega_X^d)^{\otimes r} \to \sO_{\overline{X}}(rK_{\overline{X}})$ coincides with $\mathrm{Jac}_f^r \cdot \sO_{\overline{X}}(rK_{\overline{X}})$.
Consider the decomposition $r(K_{\overline{X}}-f^*(K_X+\Delta))=K_{+}-K_{-}$, where $K_+, K_-$ are effective divisors on $\overline{X}$. 
Then we have the following commutative diagram: 
\[
\xymatrix{
f^*(\Omega_X^d)^{\otimes r} \otimes \sO_{\overline{X}}(-K_-) \ar[r] \ar[dr] & \sO_{\overline{X}}(rK_{\overline{X}})\\
& f^*\sO_X(r(K_X+\Delta)) \otimes \sO_{\overline{X}}(-K_-). \ar@{^{(}->}[u]
}
\]
Computing the images of these maps, we see that 
$$\mathrm{Jac}_f^r \cdot \sO_{\overline{X}}(-K_-)=\mathcal{I}_{r, \Delta}\sO_{\overline{X}}(-K_+),$$
which gives the assertion. 
\end{proof}

Now we state a subadditivity formula for multiplier ideals involving the ideal sheaf $\mathcal{I}_{r, \Delta}$. 
\begin{thm}\label{main2}
Let $X$ be a normal variety over an algebraically closed field of characteristic zero and $\Delta$ be an effective $\Q$-divisor on $X$ such that $r(K_X+\Delta)$ is Cartier for an integer $r \ge 1$. 
Let $\mathcal{I}_{r, \Delta}$ be the ideal sheaf given in Definition \ref{char 0 def}. 
Then
$$\J\left((X, \Delta); \a^s\b^t\mathcal{I}_{r, \Delta}^{1/r}\right) \subseteq \J((X, \Delta); \a^s)\J((X, \Delta); \b^t)$$
for any ideal sheaves $\a, \b \subseteq \sO_X$ and for any real numbers $s, t>0$. 
In particular, 
\begin{equation}\tag{$\star$}
\overline{\mathrm{Jac}_X} \cdot \J\left((X, \Delta); \a^s\b^t \mathcal{I}_X^{(\bullet)}\right) \subseteq \J((X, \Delta); \a^s)\J((X, \Delta); \b^t),
\end{equation}
where $\overline{\mathrm{Jac}_X}$ is the integral closure of the Jacobian ideal sheaf of $X$ and 
$\mathcal{I}_{\Delta}^{(\bullet)}=\{\sO_X(-\lceil m \Delta \rceil)\}$ is the graded family of ideal sheaves associated to $\Delta$. 
\end{thm}
\begin{proof}
Let $p_1, p_2 :X \times X \to X$ be the natural projections. 
We regard $X$ as a closed subvariety of $X \times X$ via the diagonal embedding $X \hookrightarrow X \times X$. 
Since 
$$J((X, \Delta); \a^s)\J((X, \Delta); \b^t)=\J((X \times X, p_1^*\Delta+p_2^*\Delta); (p_1^{-1}\a)^s (p_2^{-1}\b)^t)\big|_X,$$
it is enough to show that
$$\J\left((X, \Delta); \a^s\b^t \mathcal{I}_{r, \Delta}^{1/r}\right) \subseteq \J\left((X \times X, p_1^*\Delta+p_2^*\Delta); (p_1^{-1}\a)^s (p_2^{-1}\b)^t\right)\big|_X.$$
Let $\pi: \widetilde{X} \to X$ be a log resolution of $\Delta, \a, \b$ and $\mathcal{I}_{r, \Delta}$ so that $\mathcal{I}_{r, \Delta}\sO_{\widetilde{X}}=\sO_{\widetilde{X}}(-F_{r, \Delta})$, 
and denote by $g=\pi \times \pi: \widetilde{X} \times \widetilde{X} \to X \times X$  the product morphism. 
Note that the restriction of $g$ to the diagonal is nothing but $\pi$. 
Let $h: Y \to \widetilde{X} \times \widetilde{X}$ be a morphism such that the composite morphism $f=(h \circ g): Y \to X \times X$ is a factoring embedded resolution of $X$ in $X \times X$ and $\mathcal{I}_X \sO_{Y}=\mathcal{I}_{\overline{X}}\sO_Y(-R_X)$, where $\overline{X}$ is the strict transform of $X$ in $Y$. 
\[\xymatrix{
Y \ar[r]^{h \quad} \ar@/^1.7pc/[rr]^f & \widetilde{X} \times \widetilde{X} \ar[r]^g & X \times X\\
\overline{X} \ar@{^{(}->}[u] \ar[r]^{h|_{\overline{X}}} & \widetilde{X} \ar@{^{(}->}[u] \ar[r]^{\pi
} & X \ar@{^{(}->}[u]\\
}\]
Put $B=K_Y-f^*(K_{X \times X}+p_1^*\Delta+p_2^*\Delta)$ and denote $d$ by the dimension of $X$. 
It then follows from Lemma \ref{surjective} that 
\begin{align*}
&{f|_{\overline{X}}}_*\sO_{\overline{X}}(\lceil B-s \cdot f^{-1}(V(p_1^{-1}\a))-t \cdot f^{-1} (V(p_2^{-1}\b))-d \cdot R_X \rceil |_{\overline{X}})\\
=&f_*\sO_Y(\lceil B-s \cdot f^{-1}(V(p_1^{-1}\a))-t \cdot f^{-1}(V(p_2^{-1}\b))-d \cdot R_X \rceil)\big|_X\\
\subseteq& f_*\sO_Y(\lceil B-s \cdot f^{-1}(V(p_1^{-1}\a))-t \cdot f^{-1}(V(p_2^{-1}\b)) \rceil) \big|_X\\
=&\J((X \times X, p_1^*\Delta+p_2^*\Delta); (p_1^{-1}\a)^s (p_2^{-1}\b)^t) \big|_X.
\end{align*}

\begin{cl}
$$(B- d \cdot R_X)\big|_{\overline{X}} \ge K_{\overline{X}}-{f\big|_{\overline{X}}}^*(K_X+\Delta)-\frac{1}{r} \cdot h\big|_{\overline{X}}^*F_{r, \Delta}.$$
\end{cl}

\begin{proof}[Proof of Claim]
Applying Lemma \ref{Jacobian} to $h, \pi$ and ${f|_{\overline{X}}}$, one has
\begin{align*}
\mathrm{Jac}_h&=\sO_{Y}(-K_{Y/\widetilde{X} \times \widetilde{X}}), \\
\mathrm{Jac}_\pi^r&=\sO_{\widetilde{X}}(-r(K_{\widetilde{X}}-\pi^*(K_X+\Delta))-F_{r, \Delta}), \\
\mathrm{Jac}_{f|_{\overline{X}}}^r&=\sO_{\overline{X}}(-r(K_{\overline{X}}-{f|_{\overline{X}}}^*(K_X+\Delta))-h|_{\overline{X}}^*F_{r, \Delta}).
\end{align*}
It follows from \cite[Lemma 6.3]{Ei} ($X$ is assumed to be normal and $\Q$-Gorenstein in \textit{loc.~cit.}, but the same statement holds when $X$ is only normal) that 
$$\mathrm{Jac}_h \big|_{\overline{X}} \left(\mathrm{Jac}_\pi \cdot \sO_{\overline{X}}\right)^2 \subseteq \mathrm{Jac}_{f|_{\overline{X}}}\cdot \sO_{\overline{X}}(-d \cdot R_X \big|_{\overline{X}}), $$
which is equivalent to saying that 
\begin{align*}
&-K_{Y/\widetilde{X} \times \widetilde{X}}\big|_{\overline{X}}-2 \cdot {h\big|_{\overline{X}}}^*K_{\widetilde{X}}+2 \cdot {f\big|_{\overline{X}}}^*(K_X+\Delta)-\frac{2}{r} \cdot h\big|_{\overline{X}}^*F_{r, \Delta}\\
\le &-K_{\overline{X}}+{f\big|_{\overline{X}}}^*(K_X+\Delta)-\frac{1}{r} \cdot h\big|_{\overline{X}}^*F_{r, \Delta}-d \cdot R_X\big|_{\overline{X}}.
\end{align*}
Note that $(K_{\widetilde{X} \times \widetilde{X}}-g^*(K_{X \times X}+p_1^*\Delta+p_2^*\Delta))\big|_{\widetilde{X}}=2 (K_{\widetilde{X}}-\pi^*(K_X+\Delta))$. 
Thus, substituting this equality to the above inequality, one has
$$K_{\overline{X}}-{f\big|_{\overline{X}}}^*(K_X+\Delta)-\frac{1}{r} \cdot h\big|_{\overline{X}}^*F_{r, \Delta} \le (B-d \cdot R_X)\big|_{\overline{X}}.$$
\end{proof}

By the above claim, we have
\begin{align*}
& \J\left((X, \Delta);\a^s \b^t \mathcal{I}_{r, \Delta}^{1/r}\right)\\
=& {f\big|_{\overline{X}}}_*\sO_{\overline{X}}(\lceil K_{\overline{X}}-{f\big|_{\overline{X}}}^*(K_X+\Delta)-s \cdot f\big|_{\overline{X}}^{-1}(V(\a))- t \cdot f\big|_{\overline{X}}^{-1}(V(\b))-\frac{1}{r} \cdot h\big|_{\overline{X}}^*F_{r, \Delta} \rceil)\\
\subseteq &{f|_{\overline{X}}}_*\sO_{\overline{X}}(\lceil B-s \cdot f^{-1}(V(p_1^{-1}\a))-t \cdot f^{-1}(V( p_2^{-1}\b))-d \cdot R_X \rceil |_{\overline{X}})\\
\subseteq & \J((X \times X, p_1^*\Delta+p_2^*\Delta); (p_1^{-1}\a)^s (p_2^{-1}\b)^t)\big|_X.
\end{align*}
\end{proof}

\begin{rem}
The inclusion $(\star)$ in Theorem \ref{main2} involves not the Jacobian ideal sheaf but its integral closure, and so Theorem \ref{main2} is a little bit stronger than Theorem \ref{main} in this sense. 
We don't know at the moment how to prove the inclusion $(\star)$ using big generalized test ideals.
The difficulty is illustrated in the fact that big generalized test ideals are not necessarily integrally closed (see \cite{Mc}).  
\end{rem}

\end{document}